\documentclass[12pt]
{amsart}
\usepackage{amsmath,latexsym,amsthm,amsfonts,tipa}
\usepackage{amssymb,amscd,graphpap}
\usepackage{enumerate}
\RequirePackage{hyperref}

\usepackage[a4paper,top=2.2cm,bottom=2.2cm,left=2.3cm,right=2.2cm]{geometry}

\theoremstyle{plain}

\theoremstyle
{plain}
\newtheorem{theorem}{Theorem}[section]
\newtheorem{Th}[theorem]{Theorem}
\newtheorem{proposition}[theorem]{Proposition}
\newtheorem{Ps}[theorem]{Proposition}
\newtheorem{lemma}[theorem]{Lemma}
\newtheorem{corollary}[theorem]{Corollary}
\newtheorem{question}[theorem]{Question}

\theoremstyle{definition}
\newtheorem{definition}[theorem]{Definition}
\newtheorem{Df}[theorem]{Definition}
\newtheorem{example}[theorem]{Example}
\newtheorem{remark}[theorem]{Remark}
\newtheorem{Rm}[theorem]{Remark}

\makeindex

\newcommand{\N}{\mathbb{N}}
\newcommand{\Z}{\mathbb{Z}}

\newcommand{\CC}{\mathbb{C}}

\newcommand{\BBB}{\mathcal B}

\newcommand{\F}{\mathcal F}
\newcommand{\II}{\mathcal I}

\newcommand{\QQ}{\mathcal Q}

\newcommand{\BBBB}{\mathbb C}

\newcommand{\K}{\mathbb M}

\DeclareMathOperator{\Spec}{Spec}

\DeclareMathOperator{\dsc}{dsc}
\DeclareMathOperator{\Sym}{Sym}
\DeclareMathOperator{\supp}{supp}

\author{D. Dikranjan, I. Protasov, K. Protasova, N. Zava}\thanks{The first named author thankfully acknowledges partial financial support via the grant PRID at the Department of mathematical, Computer and Physical Sciences, Udine University}
\title{Balleans, hyperballeans and ideals}
\date{\today}
\address{Department of mathematical, Computer and Physical Sciences, Udine University, 33 100 Udine, Italy}
\email{dikran.dikranjan@uniud.it}
\address{Department of Computer Science and Cybernetics, Kyiv University, Volodymyrska 64, 01033, Kyiv, Ukraine}
\email{i.v.protasov@gmail.com; }
\address{Department of Computer Science and Cybernetics, Kyiv University, Volodymyrska 64, 01033, Kyiv, Ukraine}
\email{ksuha@freenet.com.ua}
\address{Department of mathematical, Computer and Physical Sciences, Udine University, 33 100 Udine, Italy}
\email{nicolo.zava@gmail.com}
\begin{document}
\maketitle

\begin{abstract} A ballean $\mathcal{B}$  (or a coarse structure) on a set $X$ is a family of subsets of $X$ called balls (or entourages of the diagonal in $X\times X$) defined in such a way that $\mathcal{B}$ can be considered as the asymptotic counterpart of a uniform topological space. The aim of this paper is to study two concrete balleans defined by the ideals in the Boolean algebra of all subsets of $X$ and their hyperballeans, with particular emphasis on their connectedness structure, more specifically the number of their connected components.  
\vspace{2 mm}
	
\vspace{2mm}
	
\end{abstract}

\section{Introduction}

\subsection{Basic definitions}

A {\em ballean} is a triple $\BBB=(X,P,B)$ where $X$ and $P$ are sets, $P\neq\emptyset$,  and $B:X\times P \to \mathcal P(X)$ is a map, with the following properties:
\begin{enumerate}[(i)]
\item $x\in B(x,\alpha)$ for every $x\in X$ and every $\alpha\in P$;
\item {\em symmetry}, i.e., for any $\alpha\in P$ and every pair of points $x,y\in X$, $x\in B(y,\alpha)$ if and only if $y\in B(x,\alpha)$;
\item {\em upper multiplicativity}, i.e., for any $\alpha,\beta\in P$, there exists a $\gamma\in P$ such that, for every $x\in X$, $B(B(x,\alpha),\beta)\subseteq B(x,\gamma)$, where $B(A,\delta)=\bigcup\{B(y,\delta)\mid y\in A\}$, for every $A\subseteq X$ and $\delta\in P$.
\end{enumerate}
The set $X$ is called {\em support of the ballean}, $P$ -- {\em set of radii}, and $B(x,\alpha)$ -- {\em ball of centre $x$ and radius $\alpha$}.

This definition of ballean does not coincide with, but it is equivalent to the usual one (see \cite{ProZar} for details).

A ballean $\mathcal{B}$   is called {\em connected} if, for any $x,y \in X$, there exists $\alpha\in P$  such that $y\in B(x,\alpha)$. 
Every ballean $(X,P,B)$ can be partitioned in its {\em connected components}{$ \,$}: the {\em connected component} of a point $x\in X$ is
$$
\QQ_X(x)=\bigcup_{\alpha\in P}B(x,\alpha).
$$
Moreover, we call a subset $A$ of a ballean $(X,P,B)$ {\em bounded} if there exists $\alpha\in P$ such that, for every $y\in A$, $A\subseteq B(y,\alpha)$. The empty set is always bounded. A ballean is {\em bounded} if its support is bounded. In particular, a bounded ballean is connected. Denote by $\flat(X)$ the family of all bounded subsets of a ballean $X$. 

 If $\BBB=(X,P,B)$ is a ballean and $Y$ a subset of $X$, one can define the {\em subballean $\BBB\restriction_Y=(Y,P,B_Y)$ on $Y$ induced by $\BBB$}, where $B_Y(y,\alpha)=B(y,\alpha)\cap Y$, for every $y\in Y$ and $\alpha\in P$.

A subset $A$ of a ballean $(X,P,B)$ is {\em thin} (or {\em pseudodiscrete}) if, for every $\alpha\in P$, there exists a bounded subset $V$ of $X$ such that $B_{A}(x,\alpha)=B(x,\alpha)\cap A=\{x\}$ for each $x\in A\setminus V$. A ballean is {\em thin} if its support is thin. Bounded balleans are obviously thin. 

We note that to each ballean on a set $X$ can be associated a {\em coarse structure} \cite{Roe}: a particular family $\mathcal E$ of subsets of $X\times X$, called {\em entourages of the diagonal} $\Delta_X$. The pair $(X,\mathcal E)$ is called a {\em coarse space}. This construction highlights the fact that balleans can be considered as asymptotic counterparts of uniform topological spaces. For a categorical look at the balleans and coarse spaces as ``two faces of the same coin'' see \cite{DikZa}.

\begin{definition}[\cite{ProZar,DikZav}]\label{def:sizes} Let $\BBB=(X,P,B)$ be a ballean. A subset $A$ of $X$ is called:
	\begin{enumerate}[(i)]
		\item {\em large in $X$} if there exists $\alpha\in P$ such that $B(A,\alpha)=X$;
		\item {\em thick in $X$} if, for every $\alpha\in P$, there exists $x\in A$ such that $B(x,\alpha)\subseteq A$; 
		\item {\em small in $X$} if, for every $\alpha\in P$, $X\setminus B(A,\alpha)$ is large in $X$.	
	\end{enumerate}
\end{definition}

Let $\BBB_X=(X,P_X,B_X)$ and $\BBB_Y=(Y,P_Y,B_Y)$ be two balleans. Then a map $f\colon X\to Y$ is called
\begin{enumerate}
[(i)]
\item {\em coarse} if for every radius $\alpha\in P_X$ there exists another radius $\beta\in P_Y$ such that $f(B_X(x,\alpha))\subseteq B_Y(f(x),\beta)$ for every point $x\in X$;
\item {\em effectively proper} if for every $\alpha\in P_Y$ there exists a radius $\beta\in P_X$ such that 
$$
f^{-1}(B_Y(f(x),\alpha))\subseteq B_X(x,\beta)\ \mbox{ for every }\ x\in X;
$$
\item a {\em coarse embedding} if it is both coarse and effectively proper;
\item an {\em asymorphism} if it is bijective and both $f$ and $f^{-1}$ are coarse or, equivalently, $f$ is bijective and both coarse and effectively proper;
\item an {\em asymorphic embedding} if it is an asymorphism onto its image or, equivalently, if it is an injective coarse embedding;
\item a {\em coarse equivalence} if it is a coarse embedding such that $f(X)$ is large in $\BBB_Y$.
\end{enumerate}

We recall that a family $\mathcal{I}$ of subsets of a set $X$ is an {\it ideal} if $A, B\in \mathcal{I}$, $C\subseteq A$ imply $A\cup B\in \mathcal{I}$, $C\in\mathcal{I}$. 
In this paper, we always impose that $X\notin \mathcal{I}$ (so that $\mathcal{I}$ is proper) and $\II$ contains the ideal $\mathfrak{F}_{X}$ of all finite subsets of $X$. 
Because of this setting, 
a set $X$ that admits an ideal $\II$ is infinite, as otherwise $X\in\II$.  

We consider the following two balleans with support $X$  determined  by $\mathcal{I}$.

\begin{Df}
	\begin{enumerate}[(i)]
		\item{} The {\em $\mathcal{I}$-ary ballean} $X_{\mathcal{I}\mbox{-}ary}=(X, \mathcal{I}, B_{\mathcal{I}\mbox{-}ary})$, with radii set $\mathcal{I}$ and balls defined by 
$$
B_{\mathcal{I}\mbox{-}ary}(x,A) = \{x\} \cup A, \mbox{ for }x\in X\mbox{ and }A \in \mathcal{I};
$$ 
		\item{} The {\em point ideal  ballean} $X_{\mathcal{I}}= (X, \mathcal{I} ,B_{\mathcal{I}})$, where 
$$
	B_{\mathcal{I}}(x,A) =\begin{cases}
		\begin{aligned}
		&\{x\}  &\text{if $x\notin  A$,}\\
		&\{x\}\cup A= A &\text{otherwise.}\end{aligned}\end{cases}
$$ 
	\end{enumerate}
	
	The balleans $X_{\mathcal{I}\mbox{-}ary}$ and $X_{\mathcal{I}}$ are connected and unbounded. While $X_{\mathcal{I}}$ is thin,  $X_{\mathcal{I}\mbox{-}ary}$ is never thin (this follows from Proposition \ref{ex:idX} and results from \cite{ProZar} reported in Theorem \ref{theo:thin}). 
	
For every connected unbounded ballean $\BBB$ with support $X$ one can define the {\em satellite ballean} $X_{\mathcal{I}}$, where $\mathcal{I}=\flat(X)$  is the ideal of all bounded subsets of $X$.
\end{Df}

\begin{Ps}\label{ex:idX}
For every ideal $\II$ on a set $X$, the map $id_X\colon X_\II\to X_{\II\mbox{-}ary} $ is coarse, but it is not effectively proper. 
\end{Ps}
\begin{proof} Pick an arbitrary non-empty element $F\in\II$. Since $\II$ is a proper ideal, for every $K\in\II$, there exists $x_K\in X\setminus(F\cup K)$. Hence, in particular,
$$
B_{\II\mbox{-}ary}(x_K,F)=\{x_K\}\cup F\not\subseteq \{x_K\}=B_\II(x_K,K).
$$
\end{proof}

Let $\BBB=(X,P,B)$ be a ballean. Then the radii set $P$ can be endowed with a preorder $\leq_\BBB$ as follows: for every $\alpha,\beta\in P$, $\alpha\leq_\BBB\beta$ if and only if $B(x,\alpha)\subseteq  B(x,\beta)$, for every $x\in X$. A subset $P^\prime\subseteq P$ is {\em cofinal} if it is cofinal in this preorder (i.e., for every $\alpha\in P$, there exists $\alpha^\prime\in P^\prime$, such that $\alpha\leq_\BBB\alpha^\prime$). If $P^\prime$ is cofinal, then $\BBB=(X,P^\prime,B^\prime)$, where $B^\prime=B\restriction_{X\times P^\prime}$. If $\II$ is an ideal on a set $X$, then both the preorders $\leq_{\BBB_\II}$ and $\leq_{\BBB_{\II\mbox{-}ary}}$ on $\II$ coincide with the natural preorder $\subseteq$ on $\II$, defined by inclusion.

\begin{remark}\label{rem:asy}
Let $(X,P_X,B_X)$ and $(Y,P_Y,B_Y)$ be two balleans and $f\colon X\to Y$ be an injective map. We want to give some sufficient conditions that implies the effective properness of $f$.

(i) Suppose that there exist two cofinal subsets of radii $P^\prime_X$ and $P^\prime_Y$ of $P_X$ and $P_Y$, respectively, and a bijection $\psi\colon P_X^\prime\to P_Y^\prime$ such that, for every $\alpha\in P_X^\prime$ and every $x\in X$,
\begin{equation}\label{eq:rem_asy}
f(B_X(x,\alpha))=B_Y(f(x),\psi(\alpha))\cap f(X).
\end{equation}
We claim that, under these hypothesis, $f$ is a coarse embedding and then $f\colon X\to f(X)$ is an asymorphism. 

First of all, let us check that $f$ is coarse. Fix a radius $\alpha\in P_X$ and let $\alpha^\prime\in P_X^\prime$ such that $\alpha\leq\alpha^\prime$. Hence 
$$
f(B_X(x,\alpha))\subseteq f(B_X(x,\alpha^\prime))\subseteq B_Y(f(x),\psi(\alpha^\prime)),
$$ 
for every $x\in X$, where the last inclusion holds because of \eqref{eq:rem_asy}. As for the effective properness, since $f$ is bijective, \eqref{eq:rem_asy} is equivalent to
$$
B_X(x,\alpha)=f^{-1}(B_Y(f(x),\psi(\alpha))),
$$
for every $x\in X$, and this yelds to the thesis. In fact, for every $\beta\in P_Y$, there exists $\alpha^\prime\in P_X^\prime$ such that $\beta\leq\psi(\alpha^\prime)$ and thus, for every $x\in X$,
$$
f^{-1}(B_Y(f(x),\beta))\subseteq f^{-1}(B_Y(f(x),\psi(\alpha^\prime)))=B_X(x,\alpha^\prime).
$$

(ii) Note that $f\colon X\to Y$ is a coarse embedding if and only if $f\colon X\to f(X)$ is a coarse embedding, where $f(X)$ is endowed with the subballean structure inherited by $Y$. Suppose that $P^\prime_X\subseteq P_X$ and $P^{\prime\prime}_{f(X)}\subseteq P_Y$ are cofinal subsets of radii in $X$ and $f(X)$, respectively, and $\psi\colon P^\prime_X\to P^{\prime\prime}_{f(X)}$ is a bijection such that \eqref{eq:rem_asy} holds for every $x\in X$. Then $f$ is a coarse embedding. 

(iii) In notations of item (ii), in order to show that $P^{\prime\prime}_{f(X)}$ is cofinal in $f(X)$, it is enough to provide a cofinal subset of radii $P_Y^\prime\subseteq P_Y$ in $Y$ and a bijection $\varphi\colon P^{\prime\prime}_{f(X)}\to P_Y^\prime$ such that, for every $y\in f(X)$ and every $\alpha\in P_{f(X)}^{\prime\prime}$, $B_Y(y,\alpha)\cap f(X)=B_Y(y,\varphi(\alpha))\cap f(X)$.
\end{remark}

\subsection{Hyperballeans}

\begin{definition}
Let $\BBB=(X,P,B)$ be a ballean. Define its {\em hyperballean} to be $\exp(\BBB)=(\mathcal P(X),P,\exp B)$, where, for every $A\subseteq X$ and $\alpha\in P$,
\begin{equation}\label{hyper}
\exp B(A,\alpha)=\{C\in\mathcal P(X)\mid A\subseteq B(C,\alpha),\,C\subseteq B(A,\alpha)\}.
\end{equation}
\end{definition}

It is not hard to check that this  defines actually a ballean. Another easy observation is the following: for every ballean $(X,P,B)$, $\QQ_{\exp X}(\emptyset)=\{\emptyset\}$ and, in particular $\exp B(\{\emptyset\},\alpha)=\{\emptyset\}$ for every $\alpha\in P$, since $B(\emptyset,\alpha)=\emptyset$.  Motivated by this, we shall consider also the subballean $\exp^\ast(X)= \exp(X)\setminus \{\emptyset\}$. 

If $\BBB=(X,P,B)$ is a ballean, the subballean $X^\flat$ of $\exp\BBB$ having as support the family of all non-empty bounded subsets of $\BBB$ was already defined and studied in \cite{ProPro}. Note that, $\BBB$ is connected (resp., unbounded) if and only if $\BBB^\flat$ is connected (resp., unbounded). 

In the sequel we focus our attention on four hyperballeans defined by an ideal $\II$ on a set $X$. In particular, we investigate $\exp X_\II$ and $\exp X_{\II\mbox{-}ary}$, as well as their subballeans $X_{\II\mbox{-}ary}^{\flat}$ and $X_\II^{\flat}$.

So $\exp X_\mathcal I=(\mathcal P(X),\mathcal I,\exp B_\II)$, and, according to \eqref{hyper}, for $A\subseteq X$ and $K\in \mathcal I$ one has 

\begin{equation}\label{14sett}
\exp B_\II(A,K)=\begin{cases}\begin{aligned}
&\{(A\setminus K)\cup Y\mid\emptyset\neq Y\subseteq K\}&\text{if $A\cap K\neq\emptyset$,}\\
&\{A\}&\text{otherwise.}\end{aligned}\end{cases}
\end{equation}

 In fact, fix $C\in\exp B_\II(A,K)$. If $A\cap K=\emptyset$, then $C=A$ (as $B_\II(A,K)=A$ and, for every $A^\prime\subseteq A$, $B_\II(A^\prime,K)=A^\prime$). Otherwise, $C\subseteq B_\II(A,K)=A\cup K$. Moreover, $A\subseteq B_\II(C,K)$ if and only if $C\cap K\neq\emptyset$ and $A\subseteq C\cup K$. In other words,
$$
\exp B_\II (A,K) =\{Z\in \mathcal P(X)\mid A\setminus K \subsetneq Z \subseteq A\cup K \}, \ \ \  \text{if $A\cap K\neq\emptyset$}.
$$

Let us now compute the balls in $\exp X_{\II\mbox{-}ary}=(\mathcal P(X), \mathcal{I},\exp B_{\II\mbox{-}ary})$. As mentioned above, $\exp B_{\II\mbox{-}ary}(\{\emptyset\},K)=\{\emptyset\}$ for every $K\in\II$. Fix now a non-empty subset $A$ of $X$ and a radius $K\in\II$. Then a non-empty subset $C\subseteq X$ belongs to $\exp B_{\II\mbox{-}ary}(A,K)$ if and only if 
$$
C\subseteq B_{\II\mbox{-}ay}(A,K)=A\cup K\quad\mbox{and}\quad A\subseteq B_{\II\mbox{-}ary}(C,K)=C\cup K,
$$
since both $A$ and $C$ are non-empty. Hence
\begin{align*}
\exp B_{\II\mbox{-}ary}(A,K)&\,=\{(A\setminus K)\cup Y\mid Y\subseteq A,\,(A\setminus K)\cup Y\neq\emptyset\} =\\
&\,= \{Z\in \mathcal P(X)\mid A\setminus K \subseteq Z \subseteq A\cup K,\,Z\neq\emptyset\} 
\end{align*}
for every $\emptyset\neq A\subseteq X$ and $K\in\mathcal I$.

By putting all together, one obtains that, for every $A\subseteq X$ and every $K\in\II$,

\begin{equation}\label{14sett*}
\exp B_{\II\mbox{-}ary}(A,K)=\begin{cases}
\begin{aligned}
&\{Z\in\mathcal P(X)\mid A\setminus K\subseteq Z\subseteq A\cup K,\,Z\neq\emptyset\}&\text{if $A\neq\emptyset$,}\\
&\{A\}&\text{otherwise $A=\emptyset$.}\end{aligned}
\end{cases}
\end{equation}

\begin{remark}\label{Bool} Denote by $\BBBB_X:=\{0,1\}^X$ the  Boolean ring  of all function $X \to \{0,1\}= \Z_2$ and for $f\in \BBBB_X$ let $\supp f = \{x\in X\mid f(x) = 1\}$. Then one has a ring isomorphism $\jmath=\jmath_X\colon \mathcal P(X)\to \BBBB_X$, sending $A\in\mathcal P(X)$ to its characteristic function $\chi_A\in \BBBB_X$, so $\jmath(\emptyset) = 0$,  the zero function. Using $\jmath$, one can transfer the ball structure from $\exp B_{\II\mbox{-}ary}$ to $\BBBB_X$: for $0\ne f\in \{0,1\}^X$ and $A \in \mathcal I$ one has 
\begin{equation}\label{14sett0}
\jmath(\exp B_{\II\mbox{-}ary}(\jmath^{-1}(f),A))= \{g\mid g(x)= f(x), \  x\in X\setminus A\}=\{g\mid g\restriction_{X\setminus A}= f\restriction_{X\setminus A}\}. 
\end{equation}
While, according to (\ref{14sett}) and (\ref{14sett*}), the empty set is ``isolated" in both balleans $\exp X_\mathcal I$ and $X_{\II\mbox{-}ary}$, 
the set $\{g\mid g(x)= 0, \  x\in X\setminus A\}=\{g\mid g[X\setminus A]= \{0\}\}$ (i.e., the functions $g$ with $\supp g \subseteq A$), 
 still makes sense and seems a more natural candidate for a ball of radius $A$ centered at the zero function. 

Taking into account this observation, we modify the ballean structure on $\BBBB_X$, denoting by $\CC(X,\II)$ the new ballean, 
with balls defined by the unique formula suggested by (\ref{14sett0}): 
\begin{equation}\label{15sett}
B_{\CC(X,\II)}(f,A) := \{g\mid g(x)= f(x), \  x\in X\setminus A\}=\{g\mid g\restriction_{X\setminus A}= f\restriction_{X\setminus A}\}, 
\end{equation}
where $A\in \mathcal I$; when no confusion is possible, we shall write shortly $B_{\CC}(f,A)$. In this way 
\begin{equation}\label{16sett}
\jmath\restriction_{\exp^*(X_{\II\mbox{-}ary})}: \exp^*(X_{\II\mbox{-}ary})\to \CC(X,\II)
\end{equation}
 is an asymorphic embedding. 
The ballean $\CC(X,\II)$, as well as its subballean $\K(X,\II)$, having as support the ideal $\{g\in \BBBB_X \mid \supp g \in \mathcal I \}$
of the ring $\BBBB_X$, will play a prominent role in the paper (note that $\K(X,\II)\setminus \{\emptyset\}$ coincides with $\jmath(X_{\II-ary}^\flat)$).
\footnote{Sometimes we refer to $\CC(X,\II)$ as the $\II$-{\em Cartesian ballean}. Its 
ballean structure makes both ring operations on  $\CC(X,\II)$ coarse maps, while $\exp(X_{\II\mbox{-}ary})$
fails to have this property.
}

If $X=\mathbb{N}$ and $\II=\mathfrak F_{\N}$, then $\K(X,\II)$ is the {\em Cantor macrocube} defined in \cite{ProPro}. Motivated by this, the ballean $\K(X,\II)$, for an ideal $\II$ on a set $X$, will be called the $\mathcal I $-{\em macrocube} (o, shortly, a {\em macrocube}) in the sequel. 
\end{remark}

\begin{remark} One of the main motivations for the above definitions comes from the study of topology of hyperspaces. For an infinite discrete space $X$, 
the set  $\mathcal P(X)$ admits two standard non-discrete topologizatons via the Vietoris topology and via the Tikhonov topology.

In the case of the Vietoris topology, the local base at the point $Y\in \mathcal P(X)$ consists of all subsets of $X$ of the form 
$\{Z\in \mathcal P(X): K\subseteq  Z \subseteq Y \}$, where $K$ runs over the family of all finite subsets of  $Y$. The Tikhonov topology arises after identification of $\mathcal P(X)$ with $\{0,1\}^X$ via the characteristic functions of subsets of $ X$.
    Given an ideal $\mathcal I$ on $X$, the point ideal ballean $X_\mathcal  I$  can be considered as one of the possible asymptotic versions of the discrete space $X$, see Section 2. With these observations, one can look at $\exp X_\mathcal I$ as a counterpart of the Vietoris hyperspace of $X$, 
    and the Tikhonov hyperspace of $X$ has two counterparts $C(X,\mathcal  I)$ and $\exp  X_{\mathcal I\mbox{-}ary}$.
    These parallels are especially evident in the case of the ideal $\mathfrak F_X$ of finite subsets of $X$.
\end{remark}


\subsection{Main results}

In this paper we focus on hyperballeans
of balleans defined by means of ideals, these are the point ideal balleans and the $\II$-ary balleans. It is known (\cite{ProZar}) that the point ideal balleans are precisely the thin balleans. Inspired by this fact, in \S\ref{sec:thin}, we give some further equivalent properties (Theorem \ref{theo:thin}). 

As already anticipated the main objects of study will be the hyperballeans $\exp(X_\II)$, $\exp(X_{\II\mbox{-}ary})$
and $\CC(X,\II)$, where $\II$ is an ideal of a set $X$. By restriction, we will gain also knowledge of their subballeans $X_\II^\flat$, $X_{\II\mbox{-}ary}^\flat$ and the $\II$-macrocube $\K(X,\II)$. Since 
$X_\II$, $X_{\II\mbox{-}ary}$ and $\CC(X,\II)$ are pairwise different, it is natural to ask whether their hyperballeans are different or not. Section \S\ref{sec:3} is devoted to answering this question, comparing these three balleans from various points of view. 
In particular, we prove that $\exp(X_\II)$ and $\exp(X_{\II\mbox{-}ary})$ are different (Corollary \ref{coro:theo4.1}), although they have asymorphic subballeans (Theorem \ref{theo:3.3}), the same holds for the pair $\exp X_{\II\mbox{-}ary}$ and $\BBBB(X,\II)$. 
 Moreover, we show that $\BBBB(X,\II)$ (and in particular, $\exp^\ast(X_{\mathcal{I}{\mbox{-}ary}})$) is coarsely equivalent to a 
subballean of $\exp(X_{\mathcal{I}})$; so $\K(X,\II)$ (and in particular, $X_{\mathcal{I}\mbox{-}ary}^\flat$) is  coarsely equivalent to a subballean of $X_{\mathcal{I}}^\flat$

The final part of the section is dedicated to a special class of ideals defined as follows. For a cardinal  $\kappa$ and $\lambda\leq\kappa$ consider the ideal 
$$
\mathcal K_\lambda=\{F\subseteq\kappa\mid\lvert F\rvert<\lambda\}
$$ 
of $\kappa.$ A relevant property of this ideal is {\em homogeneity} (i.e., it is invariant under the natural action of the group $\Sym(\kappa)$ by permutations of $\kappa$ \footnote{Consequently, all these permutations become automatically asymorphisms, once we endow $\kappa$ with the point ideal ballean or the $\mathcal K_\lambda$-ary ideal structure. One can easily see that these are the only homogeneous ideals of $\kappa$.}). 
 
 For the sake of brevity denote by $\mathcal K$ the ideal $\mathcal K_\kappa$ of $\kappa$. Theorem \ref{theo:Kcubes} provides a bijective coarse embedding of a subballean of $\exp(\kappa_{\mathcal K})$ into $\exp(\kappa_{\mathcal K\mbox{-}ary})$ and, under the hypothesis of regularity of $\kappa$, also $\exp(\kappa_{\mathcal K})$ itself asymorphically embeds into $\exp(\kappa_{\mathcal K\mbox{-}ary})$.

To measure the level of disconnectedness of a ballean $\mathcal{B}$, one can consider the number $\dsc(\mathcal{B})$ of connected components of $\mathcal{B}$. Although the two hyperballeans $\exp(X_\II)$ and $\exp(X_{\II\mbox{-}ary})$ are different, they 
have the same connected components and in particular, $\dsc(\exp(X_\II))=\dsc(\exp(X_{\II\mbox{-}ary}))$. Moreover, 
this cardinal coincides with $\dsc(\BBBB(X,\mathcal{I})) + 1$
(Proposition \ref{prop:dsc}). The main goal of Section \S\ref{sec:dsc} is to compute the cardinal number $\dsc(\BBBB(X,\mathcal{I}))$. 
To this end we use a compact subspace $\mathcal{I}^\wedge$ of the Stone-\v Cech remained $\beta X \setminus X$ of the discrete space $X$.
In this terms, $\dsc(\BBBB(X,\mathcal{I})) = w(\mathcal{I}^\wedge)$. 
%

\section{Characterisation of thin connected balleans}\label{sec:thin}

 Let $\BBB=(X,P,B)$ be a bounded ballean. Then $\BBB$ is thin. Moreover, $\flat(X)=\mathcal P(X)$, while every proper subset of $X$ is non-thick and the only small subset is the empty set. Hence, we now focus on unbounded balleans. 
It is known (\cite{ProZar}) that a connected unbounded ballean $\mathcal{B}$ is thin if and only if the identity mapping of $X$ defines an asymorphism between $\mathcal{B}$ and its {\em satellite ballean}. It was also shown that these properties are equivalent to having all functions $f\colon X\to \{0,1\}$ being slowly oscillating (such a function  is called {\it slowly oscillating} if, for every $\alpha\in P$ there exists a bounded subset $V$ such that $\lvert f(B(x,\alpha))\rvert=1$ for each $x\in X\setminus V$; this is a specialisation (for $\{0,1\}$-valued functions) of the usual more general notion, \cite{ProZar}).  Theorem \ref{theo:thin} provides further equivalent properties.

For a ballean $\mathcal{B}= (X, P, B)$, we define a mapping $C\colon X\to\mathcal P(X)$ by $C(x)=X\setminus \{x\}$.

\begin{lemma}\label{lemma:thin}
Let $\mathcal{B}= (X, P, B)$ be a connected unbounded ballean. If $Y$ is a subset of $X$, then $C(Y)$ is bounded in $\exp(\mathcal B)$ if and only if there exists $\alpha\in P$ such that $\lvert B(y,\alpha)\rvert>1$, for every $y\in Y$.
\end{lemma}

\begin{proof} ($\to$) Since $C(Y)$ is bounded in $\exp(\mathcal B)$, there exists $\alpha\in P$ such that, for every $x,y\in Y$ with $x\neq y$, $C(y)\in \exp B(C(x),\alpha)$. Hence $y\in X\setminus\{x\}\subseteq B(X\setminus\{y\},\alpha)$ and $x\in X\setminus\{y\}\subseteq B(X\setminus\{x\},\alpha)$, in particular, $y\in B(Y\setminus\{y\},\alpha)$ and $x\in B(Y\setminus\{x\},\alpha)$, from which the conclusion descends.

($\gets$) Since, for every $y\in Y$, there exists $z\in Y\setminus\{y\}$ such that $y\in B(z,\alpha)$, $C(y)\in \exp B (X,\alpha)$. Hence $C(Y)\subseteq \exp B(X,\alpha)$, and the latter is bounded.
\end{proof}

If $\BBB$ is a ballean, denote by $\BBB^{\mathcal M}$  the subballean of $\exp\BBB$ whose support is the family of all non-thick non-empty  subsets of $X$. If $\BBB$ is unbounded, then so it is $\BBB^{\mathcal M}$. Moreover, $\BBB^\flat$ is a subballean of $\BBB^\mathcal M$.
 This motivation for the choice of $\mathcal M$ comes from the fact that non-thick subsets\footnote{or, equivalently, those subsets whose complement is large} were called {\em meshy} in \cite{DikZav} (this term will not be adopted here). 
 
\begin{theorem}\label{theo:thin}
Let $\BBB=(X,P,B)$ be an unbounded connected ballean. Then the following properties are equivalent:
\begin{enumerate}[(i)] 
  \item $\BBB$ is thin;
  \item $\BBB=\BBB_\II$, where $\II=\flat(X)$, i.e., $\BBB$ coincides with its satellite ballean;
  \item if $A\subseteq X$ is not thick, then $A$ is bounded;
  \item $\BBB^{\mathcal M}$ is connected;
  \item the map $C\colon X\to\mathcal P(X)$ is an asymorphism between $X$ and $C(X)$;
  \item every function $f\colon X\to \{0,1\}$ is slowly oscillating.
\end{enumerate}
\end{theorem}

\begin{proof} The implication (iii)$\to$(iv) is  trivial, since item (iii) implies that $\BBB^\mathcal M=\BBB^\flat$ and the latter is connected.  Furthermore,  (i)$\leftrightarrow$(ii) and (i)$\leftrightarrow$(vi) have already been proved in \cite{ProZar}.

(iv)$\to$(iii) Assume that $A\subseteq X$ is not thick. 
Fix arbitrarily a point $x\in X$. The singleton $\{x\}$ is bounded, hence non-thick.
By our assumption, $\BBB^{\mathcal M}$ is connected and both $A$ and $\{x\}$ are non-thick, so there must be a ball centred at $x$ and containing $A$. Therefore, $A$ is bounded.

\smallskip
(v)$\to$(i) If $\mathcal{B}$ is not thin then there is an unbounded  subset $Y$ of $X$ satisfying Lemma \ref{lemma:thin}. Since $C(Y)$ is bounded in
$\exp \mathcal{B}$, we see that $C$  is  not an asymorphism.

\smallskip
 (ii)$\to$(v) On the other hand, suppose that $\BBB=\BBB_\II$. Fix a radius $V\in\II$. Without loss of generality, suppose that $V$ has at least two elements. Now, pick an arbitrary point $x\in X$. If $x\in V$, then
$$
C(B_\II(x,V))\!=\!\{X\setminus\{y\}\mid y\in V\}\!=\!\{A\in C(X)\mid X\setminus(V\cup\{x\})\subsetneq A\subseteq X\}\!=\!\exp B_\II(C(x),V)\cap C(X).
$$
If, otherwise, $x\notin V$, then
$$
C(B_\II(x,V))\!=\!\{X\setminus\{x\}\}\!=\!\{A\in C(X)\mid (X\setminus\{x\})\setminus V\subsetneq A\subseteq X\setminus\{x\}\}\!=\!\exp B_\II(C(x),V)\cap C(X).
$$

\smallskip
(i)$\to$(iii) Suppose that $\BBB$ is thin and $A$ is an unbounded subset of $X$. We claim that $A$ is thick. Fix a radius $\alpha\in P$ and let $V\subseteq X$ be a bounded subset of $X$ such that $B(x,\alpha)=\{x\}$, for every $x\notin V$. Since $A$ is unbounded, there exists a point $x_\alpha\in A\setminus V$. Hence $B(x_\alpha,\alpha)=\{x_\alpha\}\subseteq A$, which shows that $A$ is thick.


\smallskip
(iii)$\to$(vi) Assume that $X$ does not satisfy (vi), i.e, $X$ has a non-slowly-oscillating function $f\colon X\to\{0,1\}$. Take a radius $\alpha$ such that, for every bounded $V$, there exists $x\in X\setminus V$ such that $\lvert f( B(x,\alpha))\rvert=2$. Hence $A=\{x\in X\mid\lvert f(B(x,\alpha))\rvert=2\}$ is unbounded. Decompose $A$ as the disjoint union of 
$$
A_0=\{x\in A\mid f(x)=0\}\quad\mbox{and}\quad A_1=\{x\in A\mid f(x)=1\}.
$$
Since $A=A_0\cup A_1$, either $A_0$ or $A_1$ is unbounded. Moreover, for every $x\in A$, both $A_0\cap B(x,\alpha)\neq\emptyset$ and $A_1\cap B(x,\alpha)\neq\emptyset$ and thus $A_0$ and $A_1$ are not thick.  

\end{proof}

\begin{Rm} (i) Let us see that one cannot weaken item (iii) in the above theorem by replacing ``non-thick" by the stronger property ``small". In other words, a ballean need not be thin provided that all its small subsets are bounded. To this end consider the {\em $\omega$-universal} ballean (see \cite[Example 1.4.6]{ProZar}): an infinite countable set $X$, endowed with the radii set 
$$
P=\{f\colon X\to[X]^{<\infty}\mid x\in f(x),\,\{y\in X\mid x\in f(y)\}\in [X]^{<\infty},\,\forall x\in X\},
$$
and $B(x,f)=f(x)$, for every $x\in X$ and $f\in P$. Since it is {\em maximal} (i.e., it is connected, unbounded and every properly finer ballean structure is bounded) by \cite[Example 10.1.1]{ProZar}, then every small subset is finite (by application of \cite[Theorem 10.2.1]{ProZar}), although it is not thin. 

(ii)  Let $\BBB$ be an unbounded connected ballean and $X$ be its support. Consider the map $CB\colon\BBB^\flat\to\exp\BBB$ such that $CB(A)=X\setminus A$, for every bounded $A$. It is trivial that $C=CB\restriction_X$, where $X$ is identified with the family of all its singletons. Hence, if $CB$ is an asymorphic embedding, then $C$ is an asymorphic embedding too, and thus $\BBB$ is thin, according to Theorem \ref{theo:thin}. However, we claim that $CB$ is not an asymorphic embedding if $\BBB$ is thin and then item (v) in Theorem \ref{theo:thin} cannot be replaced with this stronger property.

Since $\BBB$ is thin, we can assume that $\BBB$ coincides with its satellite $\BBB_\II$ (Theorem \ref{theo:thin}). Fix a radius $V\in\II$ of $\exp X_\II$ and suppose, without loss of generality, that $V$ has at least two elements. For every radius $W\in\II$ of $X_\II^\flat$, pick an element $A_W\in\II$ such that $A_W\subseteq X\setminus(W\cup V)$. Hence, $CB^{-1}(\exp B_\II(CB(A_W),V))\not\subseteq B_\II^\flat(A_W,W)=\{A_W\}$, which implies that $CB$ is not effectively proper. In fact, since $A_W\cup V\in\II$,
$$
\exp B_\II(CB(A_W),V)=\{Z\subseteq X\mid X\setminus(A_W\cup V)\subsetneq Z\subseteq X\setminus A_W\}\subseteq CB(X_\II^\flat),
$$ 
and thus $\lvert \exp B_\II(CB(A_W),V)\cap CB(X_\II^\flat)\rvert>1$. 
\end{Rm}

A characterization of thin (and coarsely thin) balleans in terms of asymptotically isolated balls
can be found in \cite[Theorems 1, 2]{PetPro}. 

\section{Further properties of $\exp(X_{\mathcal{I}})$, $\exp(X_{\mathcal{I}{\mbox{-}ary}})$ and $\BBBB(X,\II)$}\label{sec:3}

Let $f\colon X\to Y$ be a map between sets. Then there is a natural definition for a map $\exp f\colon\mathcal P(X)\to\mathcal P(Y)$, i.e., $\exp f(A)=f(A)$, for every $A\subseteq X$. If $f\colon X\to Y$ is a map between two balleans such that $f(A)\in\flat(Y)$, for every $A\in\flat(Y)$ (e.g., a coarse map), then the restriction $f^\flat=\exp f\restriction_{X^\flat}\colon X^\flat\to Y^\flat$ is well-defined.

The following proposition can be easily proved.
\begin{proposition}\label{prop:morphisms}
	Let $\BBB_X=(X,P_X,B_X)$ and $\BBB_Y=(Y,P_Y,B_Y)$ be two balleans and let $f\colon X\to Y$ be a map between them. Then:
	\begin{enumerate}[(i)]
		\item $f\colon\BBB_X\to\BBB_Y$ is coarse if and only if $\exp f\colon\exp\BBB_X\to\exp\BBB_Y$ is coarse if and only if $f^\flat\colon \BBB_X^\flat\to \BBB_Y^\flat$ is  well-defined and coarse;
		\item $f\colon\BBB_X\to\BBB_Y$ is a coarse embedding if and only if $\exp f\colon\exp\BBB_X\to\exp\BBB_Y$ is a coarse embedding if and only if $f^\flat\colon \BBB_X^\flat\to \BBB_Y^\flat$ is well-defined and a coarse embedding.
	\end{enumerate}
\end{proposition}

For the sake of simplicity, throughout this section, for every ideal $\II$ of a set $X$, the ballean $\BBBB(X,\II)$ will be identified with $\jmath^{-1}(\BBBB(X,\II))$ (where $\jmath$ is defined in Remark \ref{Bool}), whose support is $\mathcal P(X)$. Hence, by this identification, if $A\subseteq X$ and $K\in\II$,
$$
B_\BBBB(A,K)=\{Y\mid A\setminus K\subseteq Y\subseteq A\cup K\}.
$$

\begin{corollary}\label{coro:theo4.1}
	For every ideal $\mathcal{I}$ on $X$, the following statements hold:
	\begin{enumerate}[(i)] 
		\item  $j=\exp id_X\colon\exp X_\mathcal I\to \exp X_{\mathcal I\mbox{-}ary}$ is coarse, but it is not an asymorphism;
		\item $j\colon\exp X_{\II\mbox{-}ary}\to\BBBB(X,\II)$ is coarse, but it is not an asymorphism;
		\item the same holds for the restriction $i=id_X^\flat\colon X_\mathcal I^{\flat}\to X_{\II\mbox{-}ary}^{\flat}$. 
	\end{enumerate}
\end{corollary}

\begin{proof} 
	Since $id_X\colon X_\II\to X_{\II\mbox{-}ary}$ is coarse, but it is not effectively proper (Proposition \ref{ex:idX}), items (i) and (iii) follow from Propositions \ref{prop:morphisms}. Item (ii) descends from the fact that  $\exp(X_{\II\mbox{-}ary})\restriction_{\mathcal P(X)\setminus\{\emptyset\}}=\BBBB(X,\II)\restriction_{\mathcal P(X)\setminus\{\emptyset\}}$, and $\QQ_{\exp X_{\II\mbox{-}ary}}(\emptyset)=\{\emptyset\}$, while $\QQ_{\BBBB(X,\II)}(\emptyset)=\II$.
\end{proof}

In spite of Corollary \ref{coro:theo4.1}, we show now that a cofinal part of $\exp(X_\II)$ asymorphically embeds in $\exp(X_{\II\mbox{-}ary})$.

For every ideal $\mathcal I$ on $X$ and $x\in X$ consider the families $\mathcal U_x=\{U\subseteq X\mid x\in U\}$, the principal ultrafilter of $\mathcal P(X)$ generated by $\{x\}$, and $\II_x=\mathcal U_x\cap\II=\{F\in\II\mid x\in F\}$.

\begin{Th}\label{theo:3.3} For every ideal $\mathcal{I}$ on $X$ and $x\in X$, the following statements hold:
	\begin{enumerate}[(i)] 
		\item if $j\colon\exp(X_\II)\to\exp(X_{\II\mbox{-}ary})$ is the map defined in Corollary \ref{coro:theo4.1}(i), its restriction $j\restriction_{\mathcal U_x}$ is an asymorphism between the corresponding subballeans;
		\item $\BBBB(X,\II)$ and, in particular, $\exp^\ast(X_{\mathcal{I}{\mbox{-}ary}})$ are coarsely equivalent to the subballean of $\exp(X_{\mathcal{I}})$ with support $\mathcal{U}_{x}$, witnessed by the map  
		$$
		j\restriction_{\mathcal U_x}\colon \exp(X_\mathcal I)\restriction_{\mathcal U_x}\to\exp^\ast(X_{\II\mbox{-}ary})\subseteq\BBBB(X,\II). 
		$$
	\end{enumerate}
\end{Th}

\begin{proof} (i) For every $C\in\mathcal U_x$ and $A\in\mathcal I_x$, we have
	$$
	\exp B_\II(C,A)\cap\mathcal U_x=\{(C\setminus A)\cup Y\mid x\in Y\subseteq A\} =\exp B_{\II\mbox{-}ary}(C,A)\cap\mathcal U_x,
	$$
	since $C\cap A\neq\emptyset$. Hence the conclusion follows by Remark \ref{rem:asy}(i), since $\II_x$ is a cofinal subset of radii of $\II$.
	
	(ii) In view of item (i), it remains to see that $j(\mathcal U_x)$ is large in $\BBBB(X,\II)$. 
	Indeed, for every $A\in\mathcal U_x$, $B_{\BBBB}(A,\{x\})=\{A,A\setminus\{x\}\}$, and so $B_{\BBBB}(j(\mathcal U_x),\{x\})=\BBBB(X,\II)$, where $\{x\}\in\II$. 
\end{proof}

Since $X_\II^\flat$, $X_{\II\mbox{-}ary}^\flat$, and $\K(X,\II)$ are subballeans of $\exp(X_\II)$, $\exp(X_{\II\mbox{-}ary})$, and $\BBBB(X,\II)$ respectively, 
by taking by restrictions we obtain the following immediate corollary.

\begin{corollary}
	For every ideal $\mathcal{I}$ on $X$ and $x\in X$, the following statements hold:
	\begin{enumerate}[(i)]
		\item $j\restriction_{\mathcal I_x}$ is an asymorphism between the corresponding subballeans of $X_\II^\flat$ and $X_{\II\mbox{-}ary}^\flat$;
		\item $\K(X,\II)$ and, in particular, $X_{\mathcal{I}\mbox{-}ary}^\flat$ are coarsely equivalent to the subballean of $X_{\mathcal{I}}^\flat$ with support $\mathcal{I}_{x}$, witnessed by the map $j\restriction_{\mathcal I_x}\colon X_\mathcal I^\flat\restriction_{\mathcal I_x}\to X_{\II\mbox{-}ary}^\flat\subseteq\K(X,\II)$.
	\end{enumerate}
\end{corollary}

\subsection{$\BBBB(\kappa,\mathcal K)$ and the hyperballeans $\exp(\kappa_{\mathcal K})$ and $\exp(\kappa_{\mathcal K\mbox{-}ary})$}\label{subsec:cubes}

Now we focus our study on some more specific ideals. For an infinite cardinal $\kappa$ and for its ideal
$$
\mathcal{K}:= [\kappa]^{<\kappa}=\{Z\subset\kappa\mid\lvert Z \rvert<\kappa \},
$$ 
consider the two balleans $\kappa_{\mathcal K}$ and $\kappa_{\mathcal K\mbox{-}ary}$. Here we investigate some relationships between hyperballeans of those two balleans and the ballean $\BBBB(\kappa,\mathcal K)$. 

Furthermore, with $\kappa$ as above, if $x<\kappa$, put 
$$
\mathcal U_{\ge x}=\{A\subseteq\kappa\mid \min A=x\}\quad\mbox{and}\quad\mathcal{K}_{\geq x} =\mathcal U_{\ge x}\cap\mathcal K=\{A\in \mathcal{K}\mid\min A=x\}.
$$ 
For every pair of ordinals $\alpha\leq\beta<\kappa$, let $[\alpha,\beta] = \{\gamma\in \kappa\mid \alpha\leq\gamma\leq\beta\}$. Clearly, the cardinal $\kappa$ is regular if and only if the family $P_{int}=\{[0,\alpha]\mid\alpha<\kappa\}$ is cofinal in $\mathcal K$. For a ballean $\mathcal{B}=(X,P,B)$, two subsets $A,B$ of $X$ are called {\em close} if $A$ and $B$ are in the same connected component of $\exp\mathcal{B}$.

\begin{Th}\label{theo:Kcubes} Let $\kappa$ be an infinite cardinal and $x<\kappa$. Then: 
	\begin{enumerate}[(i)]
		\item  the subballean $\mathcal{U}_{\geq x}$ of $\exp(\kappa_{\mathcal{K}})$ is asymorphic to $\BBBB(\kappa,\mathcal K)$, so $\exp^\ast(\kappa_{\mathcal{K}})$ is the disjoint union of $\kappa$ pairwise close copies of $\BBBB(\kappa,\mathcal K)$;
		\item if $\kappa$ is regular, then $\exp(\kappa_{\mathcal K})$ asymorphically embeds into $\exp(\kappa_{\mathcal K\mbox{-}ary})$.
	\end{enumerate}
\end{Th}

\begin{proof}  
	(i) Fix a bijection $g\colon\kappa\to\mathcal A$, where $\mathcal A$ is the family of all ordinals $\alpha$ such that $x<\alpha<\kappa$. Define a map $f\colon\BBBB(\kappa,\mathcal K)\to\mathcal U_{\ge x}\subseteq\exp(\kappa_{\mathcal K})$ such that, for every $X\subseteq\kappa$, $f(X)=g(X)\cup\{x\}$. We claim that $f$ is the desired asymorphism. In order to prove it, we want to apply Remark \ref{rem:asy}(ii). Fix a radius $K\in\mathcal K$ (i.e., $\lvert K\rvert<\kappa$). Then, for every $X\subseteq\kappa$,
	\begin{equation*}
	\begin{aligned}
	f(B_{\BBBB}(X,K))&\,=f(\{Y\subseteq\kappa\mid X\setminus K\subseteq Y\subseteq X\cup K\})=\\
	&\,=\{f(g^{-1}(Z))\mid X\setminus K\subseteq g^{-1}(Z)\subseteq X\cup K\}=\\
	&\,=\{Z\cup\{x\}\mid g(X)\setminus g(K)\subseteq Z\subseteq g(X)\cup g(K)\}=\\
	&\,=\{W\in\mathcal U_{\ge x}\mid(g(X)\cup\{x\})\setminus(g(K)\cup\{x\})\subsetneq W\subseteq\\
	&\,\quad\quad\quad\quad \subseteq (g(X)\cup\{x\})\cup(g(K)\cup\{x\})\}=\\
	&\,=\exp B_\mathcal K(g(X)\cup\{x\},g(K)\cup\{x\})\cap\mathcal U_{\ge x}=\exp B_\mathcal K(f(X),g(K)\cup\{x\})\cap\mathcal U_{\ge x}.
	\end{aligned}
	\end{equation*}
	If we show that $\{g(K)\cup\{x\}\mid K\in\mathcal K\}$ is cofinal in $f(\kappa)=\mathcal U_{\ge x}$, then the conclusion follows, since we can apply Remark \ref{rem:asy}(ii) by putting $\psi(K)=g(K)\cup\{x\}$, for every $K\in\mathcal K$. It is enough to check that, for every $X\subseteq\kappa$ and $K\in\mathcal K$,
	$$
	\exp B_{\mathcal K}(X,K)\cap\mathcal U_{\ge x}=\exp B_{\mathcal K}(X,\psi(K))\cap\mathcal U_{\ge x},
	$$
	which proves the cofinality of $\psi(\mathcal K)$ in $f(\kappa)$, in virtue of Remark \ref{rem:asy}(iii).
	
	The last assertion of item (i) follows from the facts that the family $\{\mathcal{U}_{\geq x}\mid x<\kappa\}$ is a partition of  $\exp(\kappa_{\mathcal{K}})$, and, for every $x<y<\kappa$, $\mathcal U_{\ge x}\in\exp B_\mathcal K(\mathcal U_{\ge y},[x,y])$, where $[x,y]\in\mathcal K$.
	
	(ii) Every ordinal $\alpha\in\kappa$ can be written uniquely as $\alpha=\beta+n$, where $\beta$ is a limit ordinal and $n$  is a natural number. We say that $\alpha$ is {\em even} ({\em odd}) if $n$  is even (odd). We denote by $E$ the set of all odd ordinals from $\kappa$ and fix a monotonically increasing bijection $\varphi\colon\kappa\to E$. For each non-empty $F\subseteq\kappa$, let $y_F\in\kappa$ such that $y_F+1=\min\varphi(F)$ and define $f(F)=\{y_F\}\cup\varphi(F)$. Moreover, we set $f(\emptyset)=\emptyset$. Let $S=f(\exp(\kappa_{\mathcal K}))$. Hence the elements of $S$ are the empty set and those subsets $A$ of $\kappa$, consisting of odd ordinals and precisely one even ordinal $\alpha\in A$ such that $\alpha= \min A$.
	
	We claim that $f\colon \exp(\kappa_{\mathcal K})\to S$ is an asymorphism. 
	
	Since $\kappa$ is regular, $P_{int}\subseteq\mathcal K$ is a cofinal subset of radii. Now fix $[0,\alpha]\in P_{int}$. Take an arbitrary subset $A$ of $\kappa$. We can assume $A$ to be non-empty, since in that case, there is nothing to be proved. The thesis follows, once we prove that
	\begin{equation}\label{eq:theo4.2}
	f(\exp B_\mathcal K(A,[0,\alpha]))=\exp B_{\mathcal K\mbox{-}ary}(f(A),[0,\varphi(\alpha)])\cap S,
	\end{equation}
	since we can apply Remark \ref{rem:asy}(i) if we define the bijection $\psi([0,\beta])=[0,\varphi(\beta)]$, for every $\beta<\kappa$, between cofinal subsets of radii.
	
	If $A\cap [0,\alpha]=\emptyset$, then also $f(A)$ and $[0,\varphi(\alpha)]$ are disjoint, which implies that 
	$$
	\exp B_{\mathcal K\mbox{-}ary}(f(A),[0,\varphi(\alpha)])\cap S=\{f(A)\}.
	$$
	
	Otherwise, suppose that $A$ and $[0,\alpha]$ are not disjoint. In particular $y_A\in[0,\varphi(\alpha)]$. We divide the proof of \eqref{eq:theo4.2} in this case in some steps.
	
	First of all we claim that, for every $\emptyset\neq Z\subseteq \kappa$,
	\begin{equation}\label{eq:step1}
	\mbox{if $A\setminus[0,\alpha]\subseteq Z\subseteq A\cup[0,\alpha]$, then: $Z\neq A\setminus[0,\alpha]$ if and only if $y_Z\leq\varphi(\alpha)$}. 
	\end{equation}
	In fact, if $Z=A\setminus[0,\alpha]$, then $\min Z>\alpha$ and so $\min\varphi(Z)>\varphi(\alpha)$. Since $\varphi(\alpha)\in E$, $\varphi(Z)\subseteq E$, and $y_Z\notin E$, we have that $y_Z>\varphi(\alpha)$. Conversely, if $Z\neq A\setminus[0,\alpha]$, there exists $z\in Z\cap[0,\alpha]$, since $Z\subseteq A\cup[0,\alpha]$. Hence $\min Z\leq z\leq\alpha$ and thus $y_Z<\min\varphi(Z)\leq\varphi(\alpha)$.
	
	Fix now a subset $Z\subseteq\kappa$. If $f(Z)\in\exp B_{\mathcal K\mbox{-}ary}(f(A),[0,\varphi(\alpha)])$, then, by applying the definitions,
	$$
	\varphi(A)\setminus[0,\varphi(\alpha)]=f(A)\setminus[0,\varphi(\alpha)]\subseteq\varphi(Z)\cup\{y_Z\}\subseteq f(A)\cup[0,\varphi(\alpha)]=\varphi(A)\cup[0,\varphi(A)].
	$$
	Note that $\varphi(Z)\subseteq E$ and $y_Z\notin E$. Hence $$
	\varphi(A\setminus[0,\alpha])=\varphi(A)\setminus\varphi([0,\alpha])\subseteq\varphi(Z)\subseteq\varphi(A)\cup\varphi([0,\alpha])=\varphi(A\cup[0,\alpha])\quad\text{and}\quad y_Z\leq\varphi(\alpha).
	$$
	Since $\varphi$ is a bijection, we can apply \eqref{eq:step1} and obtain that $A\setminus[0,\alpha]\subsetneq Z\subseteq A\cup[0,\alpha]$, which means that $Z\in\exp B_\mathcal K(A,[0,\alpha])$. Hence we have proved the inclusion $(\supseteq)$ of \eqref{eq:theo4.2}. Since all the previous implications can be reverted, then \eqref{eq:theo4.2} finally follows.
\end{proof}

\begin{corollary}
	Let $\kappa$ be an infinite cardinal and $x<\kappa$. Then:
	\begin{enumerate}[(i)]
		\item the subballean $\mathcal{K}_{\geq x}$ of $\kappa_{\mathcal{K}}^{\flat}$ is asymorphic to $\K(\kappa,\mathcal K)$, so $\kappa_{\mathcal{K}}^{\flat}$ is the disjoint union of $\kappa$ pairwise close $\mathcal K$-macrocubes $\K(\kappa,\mathcal K)$;
		\item  if $\kappa$  is regular then $\kappa_{\mathcal{K}}^{\flat}$ asymorphically embeds into $\kappa_{\mathcal{K}\mbox{-}ary}^\flat$. 
	\end{enumerate}
\end{corollary}

The proof of item (ii), specified for $\kappa=\omega$, can be found in \cite{ProPro}.

\section{The number of connected components of $\exp (X_{\mathcal{I}})$}\label{sec:dsc}

Recall that  $\dsc (\mathcal{B})$ denotes the number of connected components of a ballean $\mathcal{B}$.
Clearly, 
\begin{equation}\label{cccc}
\dsc(\exp(\BBB))= \dsc(\exp^*(\BBB))+ 1\ge 2
\end{equation}
 for every non-empty ballean $\BBB$. We begin with the following crucial observation.

\begin{Ps}\label{prop:dsc} For an ideal $\mathcal{I}$ on  a set $X$, one has  
	\begin{enumerate}[(i)]
		\item the non-empty subsets $Y, Z$ of $X$  are close in $\exp(X_{\II-ary})$ if and only if $Y  \triangle Z \in \mathcal{I}$: 
		\item two functions $f, g\in \BBBB_X$  are close in $\CC(X, \mathcal I)$ if and only if $\supp f   \triangle \supp g \in \mathcal{I}$;
		\item for every $A\subseteq X$, $\QQ_{\exp(X_\II)}(A)=\QQ_{\exp(X_{\II-ary})}(A)$, and in particular, 
		\begin{equation}\label{Eq15}
		\begin{gathered}
		\dsc (\exp(X_{\mathcal{I}}))= \dsc (\exp(X_{\mathcal{I}{\mbox{-}ary}}))  \mbox{ and }\\ 
		\dsc (\exp^*(X_{\mathcal{I}}))=\dsc (\exp^*(X_{\mathcal{I}{\mbox{-}ary}})) =\dsc (\CC (X,{\mathcal{I}}))
		\end{gathered}
		\end{equation}
		\item $\dsc (\exp(X_{\mathcal{I}}))=\dsc (\CC (X,{\mathcal{I}})) + 1$. 
	\end{enumerate} 
\end{Ps}

\begin{proof} (i) Two non-empty subsets $Y$ and $Z$ of $\exp(X_{\II\mbox{-}ary})$ are close if and only if there exists $K\in\II$ such that $Y\in\exp B_{\II\mbox{-}ary}(Z,K)$, i.e., 
	\begin{equation}\label{eq:prop4.1}
	Y\subseteq Z\cup K\quad\text{and}\quad Z\subseteq Y\cup K.
	\end{equation}
	If \eqref{eq:prop4.1} holds, then
	$$
	Y\triangle Z=(Y\setminus Z)\cup(Z\setminus Y)\subseteq ((Z\cup K)\setminus Z)\cup((Y\cup K)\setminus Y)=K\in\II.
	$$
	Conversely, if $Y\triangle Z\in\II$,  then $K:=Y\triangle Z$ trivially satisfies \eqref{eq:prop4.1}.  
	
	(ii) Let $Y=\supp f$ and $Z=\supp g$. If both $f,g$ are non-zero, then $Y,Z$ are  non-empty and the assertion follows from (i)
	and the asymorphism between $\exp^*(X_{\II\mbox{-}ary})$ and $\BBBB(X,{\mathcal{I}})\setminus \{0\}$. 
	If $g=0$, then $f$ is close to $g$ if and only if $f\in \jmath(\II)$, i.e., $Y=\jmath^{-1}(f)\in \II$. As $Z=\emptyset$, this proves the assertion in this case as well. 
	
	(iii) Fix a subset $A$ of $X$. The inclusion  $\QQ_{\exp(X_\II)}(A)\subseteq\QQ_{\exp(X_{\II\mbox{-}ary})}(A)$ follows from Corollary \ref{coro:theo4.1}(i). 
	
	Let us check the inclusion $\QQ_{\exp(X_\II)}(A)\supseteq \QQ_{\exp(X_{\II-ary})}(A)$. If $A=\emptyset$, the claim is trivial, since $\QQ_{\exp Y}(\emptyset)=\{\emptyset\}$, for every ballean $Y$. Otherwise, fix an element $x\in A$. Let $C\in\exp B_{\II\mbox{-}ary}(A,K)$, for some $K\in\II$.  Then $C\ne \emptyset$, so we can fix also a point $y\in C$ and let $K^\prime=K\cup\{x,y\}\in\II$.  Then
	$$
	C\subseteq B_{\II\mbox{-}ary}(A,K)=A\cup K=B_\II(A,K')\quad\text{and}\quad A\subseteq B_{\II\mbox{-}ary}(C,K)=C\cup K=B_\II(C,K'),
	$$
	which shows that $C\in\exp B_{\II}(A,K^\prime)$. Hence, $ C\in \QQ_{\exp(X_\II)}(A)$. 
	
	This proves the equality $\QQ_{\exp(X_\II)}(A)=\QQ_{\exp(X_{\II-ary})}(A)$. It implies the 
	first as well as the second equality in (\ref{Eq15}). 
	To prove the last equality in (\ref{Eq15}), 
	it suffices to note that $\QQ_{\CC (X,{\mathcal{I}})}(0) = \mathcal{I}$, by virtue of (ii). Hence, 
	$\dsc(\CC (X,{\mathcal{I}})) = \dsc(\CC (X,{\mathcal{I}})\setminus \{0\})$. To conclude, use the fact that 
	$\exp^*(X_{\mathcal{I}{\mbox{-}ary}})$ is asymorphic to $\CC (X,{\mathcal{I}})\setminus \{0\}$. 
	
Item (iv) follows from (iii) and (\ref{cccc}) applied to $\BBB=\exp(X_{\mathcal{I}})$. 
\end{proof}

Proposition \ref{prop:dsc} allows us to reduce the computations of the number of connected components of all hyperballeans involved to the computation of the 
cardinal $\dsc(\CC (X,{\mathcal{I}})$. In the sequel we simply identify $\CC_X$ with the Boolean ring $\mathcal P(X)$. So that 
functions $f\in \CC_X$ are identifies with their support and 
the ideals  $\mathcal{I}$ of $X$ are simply the proper ideals of the Boolean ring
$\BBBB_X= \{0,1\}^X = \Z_2^X$,
containing $\bigoplus_X\Z_2$.  


According to Proposition \ref{prop:dsc}, the connected components of $\CC (X,\mathcal{I})$ are precisely the cosets $f +\mathcal{I}$ of the ideal $\mathcal{I}$, therefore, $\dsc(\CC (X,\mathcal{I} ) )$ coincides with the cardinality of the quotient ring $\BBBB_X/{\mathcal{I}}$: 
\begin{equation}\label{MainFormula}
\dsc(\CC (X,{\mathcal{I}}) ) =\lvert \BBBB_X/\mathcal{I}\rvert = |\mathcal P(X)/\mathcal{I}|.
\end{equation}
In particular, for every infinite set $X$ and an ideal $\II$ of $X$ one has $\dsc(\CC (X,{\mathcal{I}}) ) =2$ if and only if $\mathcal{I}$ is a maximal ideal. This is an obvious consequence of \eqref{MainFormula} as $ \lvert \BBBB_X/\mathcal{I}\rvert = 2$ if and only if the ideal $\mathcal{I}$ is maximal. 

%
%

\begin{remark}\label{last:Remark}
	The cardinality $\lvert \BBBB_X/\mathcal{I}\rvert$ is easy to compute in some cases, or to get at least an easily obtained estimate  for $ \lvert \BBBB_X/\mathcal{I}\rvert$ from above as we see now.
	To this end let 
	$$
	\iota({\mathcal{I}}):= \min\{\lambda\mid {\mathcal{I}} \   \mbox{ is an intersection of $\lambda $ maximal ideals}\}.
	$$
	Then 
	\begin{equation}\label{above}
	\lvert \BBBB_X/{\mathcal{I}}\rvert \leq \min \{2^{\iota({\mathcal{I}})},2^{\lvert X\rvert}\}. 
	\end{equation}
	Indeed, if $\mathcal{I}= \bigcap \{\mathfrak m_i\mid i < \iota({\mathcal{I}})\}$, where  $\mathfrak m_i$ are maximal ideals of $B$, then $B/\mathcal{I}$ embeds in the product $\prod_{i< \iota({\mathcal{I}})} \BBBB_X/\mathfrak m_i$ having size $\leq 2^{\iota({\mathcal{I}})}$ as $\BBBB_X/\mathfrak m_i\cong \Z_2$ for all $i$. To conclude the proof of
	(\ref{above}) it remains to note that obviously $\lvert \BBBB_X/{\mathcal{I}}\rvert \leq \lvert \BBBB_X \vert = 2^{\lvert X\rvert}$. 
	
	If $\iota({\mathcal{I}}) = n$ is finite, then $ \dsc  (\CC (X,\mathcal{I} ) )=2^{n} $. Indeed, now $\mathcal{I} = \mathfrak m_1\cap\cdots\cap \mathfrak m_n$ is a finite intersection of maximal ideals and the Chinese Remainder Theorem,  applied to the Boolean ring $\BBBB_X$ and the maximal ideals  $ \mathfrak m_1, \ldots  , \mathfrak m_n$, provides a ring isomorphism
	$$ 
	\BBBB_X/\mathcal{I} \cong \prod_{i=1}^n \BBBB/\mathfrak m_i \cong \Z_2^n. 
	$$ 
	In particular, $\lvert \BBBB_X/\mathcal{I}\rvert =  \lvert\Z_2^n\lvert = 2^n$.  By (\ref{MainFormula}), we deduce 
	\begin{equation}\label{Finitecase}
	\dsc  (\CC (X,\mathcal{I} ) )=2^{n}.
	\end{equation}
	
	Let us conclude now with another example. For every infinite set $X$ and the ideal $\mathcal{I}=\mathfrak{F}_{X}$ one has 
	$$
	\dsc(\exp(X_{\mathfrak{F}_{X}}))=\dsc(\CC (X, {\mathfrak{F}_{X}})=2^{\lvert X\rvert}.
	$$ 
	This follows from (\ref{MainFormula}) and $\lvert \mathfrak{F}_{X}\rvert =\lvert X\rvert < 2^{\lvert X\rvert}$, which implies 
	$\lvert \BBBB_X/\mathfrak{F}_{X}\rvert=\lvert \BBBB_X\rvert = 2^{\lvert X\rvert}$. 
\end{remark}

In order to obtain some estimate from {\em below} for $\lvert \BBBB_X/\mathcal{I}\rvert$, we need a deeper insight on the spectrum $\Spec \BBBB_X$ of $\BBBB_X$. Since $\BBBB_X$ is a Boolean ring, $\Spec \BBBB_X$ coincides with the space of all maximal ideals of $ \BBBB_X$, which can be identified with the  Stone--\v Cech compactification $\beta X$ when we endow $X$ with the discrete topology. As usual, 
\begin{itemize}
\item we identify the Stone--\v Cech compactification $\beta X$ with the set of all ultrafilters  on $X$;
\item the family $\{\overline{A}\mid A\subseteq X\}$, where $ \overline{A} =\{ p\in \beta X\mid  A\in p\}$, forms the base for the  topology of $\beta X$; and 
\item the set $X$ is embedded in $\beta X$ by sending $x\in X$ to the principal ultrafilter generated by $x$. 
\end{itemize}
For a filter $\varphi$  on $X$, define a closed subset $\overline{\varphi}$  of $\beta X$ as follows: 
$$
\overline{\varphi} := \bigcap\{\overline{A}\mid A\in\varphi\}.
$$
An ultrafilter $p\in \beta X$ belongs to $\overline{\varphi}$ if and only if $p$ contains the filter  $\varphi$. In other words, 
\begin{equation}\label{eq1}
{\varphi} := \bigcap\{u\mid u\in\overline{\varphi}\}.
\end{equation}

For an ideal $\mathcal{I}$ on $X$, we consider the filter $\varphi_\mathcal{I} =\{X \backslash A\mid A\in \mathcal{I}\}$, and we simply write  $\varphi$ when there is no danger of confusion. Similarly, for a filter $\varphi$ we define the ideal $\mathcal{I}_\varphi = \{X \backslash A\mid A\in \varphi \}$ and we simply write  $\mathcal{I}$ when there is no danger of confusion.

Finally, let $\mathcal{I}^{\wedge}=\overline{\varphi_\mathcal{I}},$
and note that all ultrafilters in $\mathcal{I}^{\wedge}$ are non-fixed, i.e., $ \mathcal{I}^{\wedge}\subseteq \beta X \setminus X$, as $\varphi$ is contained in the Fr\' echet filter of all co-finite sets on $X$ (since $\mathcal I \supseteq \mathfrak F_X$). Moreover, for a subset $A$ of $X$ one has
\begin{equation}\label{newUformula}
A  \in \mathcal{I} \ \mbox{ if and only if } \ A \not \in u \ \mbox{ for all } u \in  \overline{\varphi_\mathcal{I}} = \mathcal{I}^{\wedge}.
\end{equation}

As pointed out above, for any $X$ the compact space $\beta X$ coincides with the spectrum $\Spec \BBBB_X$ of the ring $\BBBB_X$. For an ideal  $\mathcal{I}$ on $X$, $\overline{\varphi_\mathcal{I}}$ is the set of ultrafilters on $X$ containing $\varphi$. For $u\in \overline{\varphi}_\mathcal{I}$ the ideal $\mathcal{I}_{u}$ is maximal and contains $\mathcal{I}$. More precisely, $\mathcal{I} = \bigcap_{u \in \mathcal{I}^{\wedge}} \mathcal{I}_u$. The maximal ideals  $\mathcal{I}_u$, when $u$ runs over $\overline{\varphi}$, bijectively correspond to the   maximal ideals of the quotient $\BBBB_X/\mathcal{I}$; in particular, $\lvert\mathcal{I}^\wedge\rvert = \lvert\Spec(\BBBB_X/\mathcal{I})\rvert$. Along with Remark \ref{last:Remark}, this gives: 

\begin{Ps}\label{OLDprop5.3} Let $\mathcal{I}$ be an ideal on set $X$. If  $\lvert\mathcal{I}^{\wedge}\rvert=n$ is finite, then $ \dsc  (\CC (X,\mathcal{I} ) )=2^{n}$. Otherwise, $ \dsc  (\CC (X,\mathcal{I} ) ) = w(\mathcal{I}^{\wedge})$. 
\end{Ps}

Here $w(\mathcal{I}^{\wedge})$ denotes the weight of the space $\mathcal{I}^{\wedge}$. The second assertion follows from 
(\ref{MainFormula}) and the equality $ w(\mathcal{I}^{\wedge}) =|\mathcal P(X)/\mathcal{I}|$, its proof can be found in \cite[\S 2]{CN}.

\begin{corollary}\label{prop5.5} Let $\mathcal{I}$ be an ideal on a countably infinite set set $X$ such that $\II^\wedge$ is infinite. 
Then $\dsc (\CC (X,\mathcal{I} ) ) = 2^{\omega}$. 
\end{corollary}

\begin{proof}
Being an infinite compact subset of $\beta X \setminus X$, $\II^\wedge$ contains a copy of $\beta \N$. Therefore, $w(\II^\wedge) =  2^{\omega}$. 
Now Proposition \ref{OLDprop5.3} applies. 
\end{proof}


In this section 
we have thoroughly investigated the number of connected components of $\exp(X_\II)$, where $\II$ is an ideal of a set $X$. This leaves open the question to  
estimate $\dsc(\exp(X))$, where $X$ is an arbitrary connected ballean.

Let $Y$ be a subballean of $X$. In particular, $\dsc(\exp(X))\ge\dsc(\exp(Y))$. If $Y$ is thin, we can apply Theorem \ref{theo:thin} so that $Y=Y_{\flat(Y)}$. The results 
from this section give a lower bound of $\dsc(\exp(Y))$, providing in this way also a lower bound for $\dsc(\exp(X))$, since
\begin{equation}\label{eq:finrem}
\dsc(\exp(X))\ge\sup\{\dsc(\exp(Z))\mid\text{$Z$ is a thin subballean of $X$}\}.
\end{equation}
Unfortunately, \eqref{eq:finrem} doesn't provide any useful information in the case when every thin subballean of $X$ is bounded. In fact, if 
$Z$ is a non-empty bounded subballean, then $\exp^\ast(Z)$ is connected and so $\dsc(\exp(Z))=2$.


\begin{thebibliography}{6}
	
	\frenchspacing
	

\bibitem{BaProReSlo}	 T. Banakh, I. Protasov, D. Repov\u s, S. Slobodianiuk, 
{\em Classifying homogeneous cellular ordinal balleans up to coarse equivalence}, preprint (arxiv: 1409.3910v2).
	
\bibitem{BaZa} T. Banakh, I. Zarichnyi, {\em  Characterizing the Cantor bi-cube in asymptotic categories}, Groups, Geometry, and Dynamics 5 (2011), 691--728.  
	
	
	\bibitem{CN} W.  Comfort, S. Negrepontis, {\em The Theory of Ultrafilters}, Grundlehren der mathematischen Wissenschaften, Band 211, Springer--Verlag,
	Berlin-Heidelberg-New York, 1974
	
	\bibitem{DikZa} D. Dikranjan, N. Zava,  {\em Some categorical aspects  of coarse structures and balleans}, Topology Appl. {225} (2017), 164-194.
	
	\bibitem{DikZav} D. Dikranjan, N. Zava, {\em Preservation and reflection of size properties of balleans}, Topology Appl. 221 (2017), 570--595. 
	
	\bibitem{Dow} A. Dow, Closures of discrete sets in compact spaces, Studia Sci. Math. Hungar. 42 (2005), no. 2, 227--234. 
	
	\bibitem{K}  K. Kunen, {\em  Set theory. An introduction to independence proofs}, Studies in Logic and Foundations of Math., vol. 102, North-Holland, Amsterdam-New York-Oxford, 1980.
	
	\bibitem{PetPro} O. Petrenko, I. Protasov, Balleans and filters, Mat. Stud. 38 (2012), no. 1, 3--11. 
	
	\bibitem{ProBan} I. Protasov, T. Banakh, {\em Ball Structures and Colorings of Groups and Graphs.}- Mat. Stud. Monogr. Ser 11, VNTL, Lviv, 2003.
	
	\bibitem{ProPro} I. Protasov, K. Protasova, {\em On hyperballeans of bounded geometry}, {\tt arXiv:1702.07941v1}.
	
	\bibitem{ProZar} I. Protasov, M. Zarichnyi, {\em General Asymptology}, 2007 VNTL Publishers, Lviv, Ukraine.
	
	\bibitem{Roe} J. Roe, \emph{Lectures on Coarse Geometry}, Univ. Lecture Ser., vol. 31, American Mathematical Society, Providence RI, 2003.
	
	\bibitem{Zav} N. Zava, {\em On $\F$-hyperballeans}, work in progress.
	
\end{thebibliography}
\end{document}